\documentclass[a4paper,11pt]{article}

\usepackage[utf8]{inputenc} 
\usepackage{enumerate}

\usepackage[T1]{fontenc}

\usepackage{hyperref}       
\usepackage{url}            
\usepackage{booktabs}       
\usepackage{amsfonts,amsmath,amssymb,amsthm}       
\usepackage{nicefrac}       
\usepackage{microtype}      
\usepackage{lipsum}
\usepackage{xcolor}
\usepackage{graphicx,float}
\usepackage{tikz}
\usepackage[square,comma,compress,numbers]{natbib}
\usepackage[big]{layaureo}
\usepackage{perpage}
\MakePerPage{footnote}

\theoremstyle{plain}
\newtheorem{theorem}{Theorem}[section]
\newtheorem{lemma}{Lemma}[section]

\newtheorem{claim}{Claim}[section]

\theoremstyle{definition}
\theoremstyle{definition}
\newtheorem*{definition}{Definition}
\newtheorem*{remark}{Remark}

\definecolor{myred}{RGB}{226,56,18}
\definecolor{myorange}{RGB}{228,139,0}
\definecolor{mygreen}{RGB}{4,215,17}
\definecolor{mygrey}{RGB}{180,180,180}

\def\Cx{\mathbb{C}}
\def\Chat{\widehat{\mathbb{C}}}

\def\ind{\mathrm{ind}}
\def\dist{\mathrm{dist}}

\begin{document}

\title{\textbf{\textsc{Multiply connected wandering domains of meromorphic functions: the pursuit of uniform internal dynamics}}}

\author{Gustavo R.~Ferreira\thanks{Email: \texttt{gustavo.r.f.95@gmail.com}}\\
  \small{School of Mathematics and Statistics, The Open University}\\
  \small{Milton Keynes, MK7 6AA, UK}
}

\maketitle

\begin{abstract}
Recently, Benini \textit{et al.} showed that, in simply connected wandering domains of entire functions, all pairs of orbits behave in the same way relative to the hyperbolic metric, thus giving us our first insight into the general internal dynamics of such domains. After proving in a recent manuscript that the same is not true for multiply connected wandering domains, a natural question is: how inhomogeneous can multiply connected wandering domains be? We give an answer to this question, in that we show that uniform dynamics inside an open subset of the domain generalises to the whole wandering domain. As an application of  this result, we construct the first example of a meromorphic function with a semi-contracting infinitely connected wandering domain.
\end{abstract}


\section{Introduction}
Let $f:\Cx\to\Chat$ be a meromorphic function, where $\Chat := \Cx\cup\{\infty\}$ denotes the Riemann sphere. The study of its iterates, undertaken first by Fatou and Julia in the 1920s, comprises the area of \textit{complex dynamics} -- a field of research that has been increasingly active since the latter half of the XX\textsuperscript{th} century. The \textit{Fatou set} of $f$, defined as
\[ F(f) := \{z\in \Cx : \text{$(f^n)_{n\in\mathbb{N}}$ is defined and normal in a neighbourhood of $z$}\}, \]
is known to be the set of ``regular'' dynamics, and its connected components, called \textit{Fatou components}, are mapped into one another by $f$. \textcolor{black}{So, if $U\subset F(f)$ is a Fatou component of $f$, $f^n(U)$ is, for all $n\in\mathbb{N}$, contained in some Fatou component $U_n$ of $f$. This} separates Fatou components into two kinds \textcolor{black}{-- those for which there exist $n > m \geq 0$ such that $U_n = U_m$, called \textit{(pre)periodic} components, and those for which all $U_n$ are distinct, called \textit{wandering domains}.}

The internal dynamics of preperiodic components have been studied for a century, going back to Fatou and Cremer \textcolor{black}{(see, for instance, \cite{Ber93})}. It is well-known, for instance, that a periodic Fatou component falls into one of five dynamically distinct types, each one having a distinct topological model. Studying the internal dynamics of wandering domains, on the other hand, is a far more recent undertaking. The first steps were taken by Bergweiler, Rippon, and Stallard in 2013 \cite{BRS13}, who described the behaviour of multiply connected wandering domains of entire functions (recall that the \textit{connectivity} of a domain $\Omega\subset \Chat$ is its number of complementary components). Their methods made extensive use of the geometric properties of multiply connected wandering domains of entire functions, but more recent work examined the internal dynamics of wandering domains in terms of the hyperbolic metric -- which is always an available tool when talking about Fatou components -- for both simply \cite{BEFRS19} and multiply connected \cite{Fer21} wandering domains.

This is the point of view that we adopt in this work as well. \textcolor{black}{Denoting the hyperbolic metric in the hyperbolic domain $\Omega\subset\Chat$ by $d_\Omega$ (see Section \ref{sec:prelim} for the relevant definitions),} the starting point is the following question: given a meromorphic function $f$ with a wandering domain $U$ and points $z, w\in U$, what happens to
\[ \text{$d_{U_n}\left(f^n(z), f^n(w)\right)$ as $n\to+\infty$?} \]

A central result of \cite{BEFRS19} is that, if $U$ and all its iterates are simply connected, then the answer is -- qualitatively -- independent of our particular choice of $z$ and $w$; in other words, all pairs of orbits behave in the same way. On the other hand, it was shown in \cite{Fer21} that if $U$ is multiply connected the answer may depend on the chosen pair of points -- and, in particular, that all possible long-term behaviours can co-exist in the same domain.

An immediate question, then, is: how complicated can this co-existence be? The observed cases in \cite{Fer21} are all ``well-behaved'', in the sense that there are dynamically defined, smooth laminations of $U$ that determine how the iterates of each pair of points behave. In particular, every non-empty open subset of $U$ contains distinct pairs of points exhibiting all the behaviours present in $U$. Here, we show that (in some sense) this uniformity is a general feature of wandering domains of meromorphic functions.

\begin{theorem} \label{thm:mix}
Let $U$ be a wandering domain of the meromorphic function $f$. Suppose that there exist a point $z_0\in U$ and a neighbourhood $V\subset U$ of $z_0$ such that one of the following properties holds for every $w\in V$:
\begin{enumerate}[(a)]
    \item $d_{U_n}\left(f^n(w), f^n(z_0)\right)\to 0$ (we say that $V$ is \emph{contracting relative to $z_0$});
    \item $d_{U_n}\left(f^n(w), f^n(z_0)\right)$ decreases to a limit $c(z_0, w) > 0$ without ever reaching it, except for a discrete (in $V$) set of points for which $f^k(w) = f^k(z_0)$ for some $k\in \mathbb{N}$ (we say that $V$ is \emph{semi-contracting relative to $z_0$});
    \item there exists $N\in\mathbb{N}$ (uniform over compact subsets of $V$) such that $d_{U_n}\left(f^n(w)\textcolor{black}{,} f^n(z_0)\right) = c(z_0, w) > 0$ for $n\geq N$ (we say that $V$ is \emph{locally eventually isometric relative to $z_0$}).
\end{enumerate}
Then, the same property holds for every $w\in U$.
\end{theorem}

\begin{remark}
\textcolor{black}{Benini \textit{et al.}'s results \cite[Theorem A]{BEFRS19} for simply connected wandering domains generalise to meromorphic functions as long as all iterates of the wandering domains are simply connected, and so the conclusion of Theorem \ref{thm:mix} always holds for such domains.} Theorem \ref{thm:mix} is also vacuous for multiply connected wandering domains with finite eventual connectivity, \textcolor{black}{which have their internal dynamics dictated by their eventual connectivity} (see \cite[Theorem 1.1]{Fer21}), and so it is of most interest for infinitely connected wandering domains of meromorphic functions.
\end{remark}

Theorem \ref{thm:mix} shows that, even in the multiply connected setting, the internal dynamics of wandering domains exhibit ``uniformity'': if something happens relative to a base point in a non-empty open set, then it happens in all of $U$.

As an application of Theorem \ref{thm:mix}, we construct a new \textcolor{black}{type of} example: a meromorphic function with an infinitely connected \textcolor{black}{semi-contracting} wandering domain. A simply connected example of such a domain, the first of its kind, was constructed by Benini \textit{et al.} using approximation theory; \cite[Theorem 1.1]{Fer21} tells us that, if a semi-contracting orbit of multiply connected wandering domains is to be found, it must be infinitely connected. With that in mind, in Section \ref{sec:example} we modify \cite[Example 2]{BEFRS19} via quasiconformal surgery to prove the following.

\begin{theorem} \label{thm:ex}
There exists a meromorphic function $f$ with an infinitely connected wandering domain $V$ and a non-empty open subset $V'\subset V$ such that, for $z_0\in V'$, $V$ is semi-contracting relative to $z_0$.
\end{theorem}

In \cite{Fer21}, the author introduced \textit{bimodal} and \textit{trimodal} wandering domains, where different long-term qualitative behaviours of the hyperbolic metric co-exist. In contrast to that, we say that a wandering domain is \textit{unimodal} if the limiting behaviour of the sequence $\left(d_{U_n}\left(f^n(z), f^n(w)\right)\right)_{n\in\mathbb{N}}$ is independent of the choice of $z$ and $w$ in $U$. Given that examples of contracting and eventually isometric examples of multiply connected wandering domains already exist (see, for instance, \cite[Example 1]{RS08} and \cite[Theorem (iii)]{BKL90}), Theorem \ref{thm:ex} completes the proof of the existence of all possible unimodal behaviours in such wandering domains. In particular, we see that the possible internal dynamics of multiply connected wandering domains (and in particular infinitely connected ones) include all possible internal dynamics that exist for simply connected wandering domains.

\noindent\textsc{Acknowledgements.} I would like to thank my supervisors, Phil Rippon and Gwyneth Stallard, for their comments and support. I am also deeply grateful to the referee for their detailed corrections and suggestions, and to Lasse Rempe for pointing out a mistake in the statement of Theorem \ref{thm:mix}.

\section{Preliminaries and notation} \label{sec:prelim}
\color{black}
We use this section to establish some notation and terminology regarding the hyperbolic metric; we refer the reader to \cite{BM06} and \cite[Chapter 3]{Hub06} for a more detailed treatment of the subject. We start with the following definition.
\begin{definition}
A continuous function $f:X\to Y$, where $X$ and $Y$ are topological spaces, is called an \textit{unbranched covering map} if:
\begin{enumerate}[(i)]
    \item $f$ is surjective;
    \item every $y\in Y$ has a neighbourhood $U_y\subset Y$ such that $f^{-1}(U_y)$ is a union of disjoint open sets $V_y\subset X$ and $f:V_y\to U_y$ is a homeomorphism for each $V_y$.
\end{enumerate}
Furthermore, if $X$ is simply connected, $f$ is called a universal covering map and $X$ is called a universal covering space of $Y$.
\end{definition}

With that in mind, we recall Koebe's uniformisation theorem (see \cite[Chapter 10]{Don11} for a modern proof):
\begin{lemma}[The Uniformisation Theorem] \label{lem:uniformisation}
Let $X$ be a Riemann surface. Then, there exists a holomorphic map $f:\tilde{X}\to X$ that is a universal covering map, and $\tilde{X}$ is exactly one of $\Cx$, $\Chat$, or $\mathbb{D} := \{z\in\Cx : |z| < 1\}$, the unit disc. Furthermore, given any $p\in X$, $f$ can be chosen so that $f(0) = p$.
\end{lemma}

In keeping with the name ``uniformisation theorem'', we will sometimes call the universal covering maps given by Lemma \ref{lem:uniformisation} \textit{uniformising maps}. We are interested in the case when the universal covering space of $X$ is $\mathbb{D}$, which happens for instance if $X$ is a domain on $\Chat$ such that $\Chat\setminus X$ contains at least three points (see \cite[Theorem 10.2]{BM06}). In this case, $X$ is called a \textit{hyperbolic surface} (or, if $X\subset\Chat$, a \textit{hyperbolic domain}), for it admits a \textit{hyperbolic metric}: a complete conformal metric of constant curvature $-1$.

Given a hyperbolic surface $X$, we will use $\rho_X$ and $d_X$ to denote its \textit{hyperbolic density} and \textit{hyperbolic distance}, respectively. Since we are following \cite{BM06,Hub06}, the hyperbolic density of the unit disc is
\[ \rho_\mathbb{D}(z) = \frac{2}{1 - |z|^2}, \]
giving us constant curvature $-1$ as promised. A curve $\gamma:[0, 1]\to X$ has \textit{hyperbolic length}
\[ \ell_X(\gamma) := \int_\gamma \rho_X(z)\,|dz|; \]
it is a consequence of the completeness of the hyperbolic metric that any two points $w, z\in X$ can be joined by a smooth curve $\gamma\subset X$ such that
\[ d_X(w, z) = \ell_X(\gamma) = \min\{\ell_X(\gamma') : \text{$\gamma'\subset X$ joins $w$ and $z$}\}, \]
and this $\gamma$ is called a \textit{hyperbolic geodesic}.

Now, let $f:X\to Y$ be a holomorphic map between hyperbolic surfaces. Its \textit{hyperbolic distortion} at $z\in X$ is given by
\[ \|Df(z)\|_X^Y := \lim_{w\to z} \frac{d_Y(f(w), f(z))}{d_X(w, z)} = \frac{\rho_Y(f(z))|f'(z)|}{\rho_X(z)}. \]
This notation refers to the fact that the hyperbolic distortion is also the norm of the differential of $f$ at $z$, viewed as a linear map from $T_zX$ to $T_{f(z)}Y$, with the metrics on the tangent spaces induced by the respective hyperbolic metrics; see \cite[Section~3.3]{Hub06}. The Schwarz-Pick lemma \cite[Theorem 10.5]{BM06} can now be expressed as follows.
\begin{lemma}[Schwarz-Pick Lemma] \label{lem:SP}
Let $f:X\to Y$ be a holomorphic map between hyperbolic Riemann surfaces. Then, for all $z\in X$,
\begin{equation} \label{eq:SPloc}
    \|Df(z)\|_X^Y \leq 1,
\end{equation}
with equality if and only if $X$ is an unbranched covering map. Additionally,
\begin{equation} \label{eq:SPglo}
    d_Y(f(z), f(w))\leq d_X(z, w)
\end{equation}
for any distinct $z$ and $w$ in $X$, with equality if and only if $f$ is biholomorphic.
\end{lemma}

Functions for which equality holds in (\ref{eq:SPloc}) are called \textit{local hyperbolic isometries (or, if the metric is clear from the context, a \textit{local isometry})}: they preserve hyperbolic distances in small neighbourhoods around each point. A function satisfying equality in (\ref{eq:SPglo}) is simply called a \textit{hyperbolic isometry}.

An immediate consequence of Lemma \ref{lem:SP} and the chain rule is that hyperbolic distortion is \textit{locally conformally invariant}. More precisely, we have (see, for instance, \cite[Theorem 10.5]{BM06}, \cite[Proposition 3.3.4]{Hub06} or \cite[Theorem 7.3.1]{KL07}):
\begin{lemma} \label{lem:hypdis}
Let $f:X\to Y$ be a holomorphic map between hyperbolic Riemann surfaces, let $z_0\in X$, let $\varphi:\mathbb{D}\to X$ be a uniformising map such that $\varphi(0) = z_0$, and let $\psi:\mathbb{D}\to Y$ be a uniformising map such that $\psi(0) = f(z_0)$. Then, there exists a holomorphic function $\tilde{f}:\mathbb{D}\to\mathbb{D}$ such that $\tilde{f}(0) = 0$ and $\psi\circ\tilde{f} = f\circ\varphi$. Furthermore, $\tilde{f}$ is unique and satisfies
\begin{equation} \label{eq:hypdis}
    |\tilde{f}'(0)| = \|Df\left(\varphi(0)\right)\|_X^Y.
\end{equation}
\end{lemma}

A function $\tilde{f}$ given by Lemma \ref{lem:hypdis} is called a \textit{lift} of $f$.

Another important aspect of hyperbolic geometry that is relevant to us are the hyperbolic discs. In the hyperbolic surface $X$, those are the open balls
\[ B_X(p, r) := \{z\in X : d_X(p, z) < r\}, \]
where $p\in X$ is the centre and $r > 0$ is the radius. These are not necessarily topologically equivalent to Euclidean open balls; an example would be to take the punctured disc $\mathbb{D}^* := \mathbb{D}\setminus\{0\}$, any point $p\in\mathbb{D}^*$, and a sufficiently large $r > 0$; then, $B_{\mathbb{D}^*}(p, r)$ surrounds the origin, and is therefore doubly connected. Thus, given $p\in X$, we must look at the \textit{injectivity radius} of $X$ at $p$: the largest value of $r > 0$ for which $B_X(p, r)$ is actually isometric to $B_\mathbb{D}(0, r)$. Another way of saying this is: the injectivity radius at $p$ is the largest $r > 0$ for which, taking a uniformising map $\varphi:\mathbb{D}\to X$ with $\varphi(0) = p$, the restriction of $\varphi$ to $B_\mathbb{D}(0, r)$ is injective (hence the name injectivity radius); see \cite[p. 166]{BP92}.

\section{Uniformity in long-term behaviours of the hyperbolic metric} \label{sec:u}
In this section, we prove Theorem \ref{thm:mix}. We divide it into three cases, each one considering a different kind of ``unimodality''.

Before doing that, though, we take some time to discuss the function $u:U\to[0, +\infty)$ defined as
\[ u(z) := \lim_{n\to+\infty} d_{U_n}\left(f^n(z), f^n(z_0)\right), \]
where we have taken some base point $z_0\in U$. It is the limit of the sequence $(u_n)_{n\in\mathbb{N}}$, where
\[ u_n(z) := d_{U_n}\left(f^n(z), f^n(z_0)\right)\text{ \textcolor{black}{for $z\in U$,}} \]
which by the Schwarz-Pick lemma (Lemma \ref{lem:SP}) satisfies
\begin{equation} \label{eq:un}
    u_1\geq u_2\geq \cdots \geq u.
\end{equation}
The following lemma gives us an inkling of how the convergence of the hyperbolic distances to their final values happens.

\begin{lemma} \label{lem:conv}
The convergence $u_n\to u$ is locally uniform. In particular, $u$ is continuous.
\end{lemma}
\begin{proof}
We will apply the Arzel\`a-Ascoli theorem \cite[Theorem 14]{Ahl79}.

First, for any $z\in U$, it is clear from the Schwarz-Pick lemma that $(u_n(z))_{n\in\mathbb{N}}$ is contained in the compact set $[0, d_U(z, z_0)]$. Second, we must show that the sequence $(u_n)_{n\in\mathbb{N}}$ is also equicontinuous on compact subsets of $U$; to this end, we apply the reverse triangle inequality to obtain
\begin{equation} \label{eq:rti}
    d_{U_n}\left(f^n(z), f^n(w)\right) \geq \left|d_{U_n}\left(f^n(z), f^n(z_0)\right) - d_{U_n}\left(f^n(w), f^n(z_0)\right)\right| = \left| u_n(z) - u_n(w)\right|
\end{equation}
for every $z, w\in U$. Now, taking any $\epsilon > 0$ and $K\subset U$ a compact set, we can choose $\delta > 0$ such that
\[ \text{$|z - w| < \delta \Rightarrow d_U(z, w) < \epsilon$ for any $z, w\in K$} \]
by the fact that $\rho_U$ is bounded above and below on $K$. Thus, by \textcolor{black}{(\ref{eq:rti}) and} the Schwarz-Pick lemma, we get for all $n\in\mathbb{N}$ that
\[ |u_n(z) - u_n(w)| \leq d_{U_n}\left(f^n(z), f^n(w)\right) \leq d_U(z, w) < \epsilon \]
whenever $|z - w| < \delta$ and $z, w\in K$, and therefore $(u_n)_{n\in\mathbb{N}}$ is equicontinuous on $K$. It follows that the Arzel\`a-Ascoli theorem \cite[Theorem 14]{Ahl79} applies\textcolor{black}{: there exists a subsequence $(u_{n_k})_{k\in\mathbb{N}}$ such that $u_{n_k}\to u$ locally uniformly. Then, (\ref{eq:un}) implies that locally uniform convergence happens for the whole sequence $(u_n)_{n\in\mathbb{N}}$.}

Finally, the locally uniform convergence of $u_n$ implies that $u$ is continuous.
\end{proof}

The continuity of $u$ will be especially important for us as we discuss how the behaviour of orbits in small neighbourhoods ``globalises'' to the whole of $U$.

\subsection{Proof of Theorem \ref{thm:mix}(a): the contracting case.}
What we prove in this section is actually stronger than what Theorem \ref{thm:mix}(a) claims; we will show that, if $U$ contains a non-empty open set $V$ such that
\[ \text{$d_{U_n}\left(f^n(z), f^n(w)\right)\to 0$ as $n\to +\infty$ for every $z, w\in V$}, \]
then the same is true for every $z, w\in U$. In order to do this, we will use results of Benini \textit{et al.} relating to non-autonomous dynamics in the unit disc (more specifically, \cite[Section 2]{BEFRS19}). As preparation for that, we \textcolor{black}{prove} the following lemma\textcolor{black}{, which tells us how to associate the dynamics of a wandering domain to a composition of inner functions. Recall that an \textit{inner function} is a holomorphic function $g:\mathbb{D}\to\mathbb{D}$ such that the radial limit
\[ g(e^{i\theta}) := \lim_{r\nearrow 1} g(re^{i\theta}) \]
exists and satisfies $|g(e^{i\theta})| = 1$ for Lebesgue almost every $\theta\in[0, 2\pi)$.}

\begin{lemma} \label{lem:lifts}
Let $f:\Cx\to\Chat$ be a meromorphic function with a multiply connected wandering domain $U$, and let $z_0$ be a point in $U$. Given uniformising maps $\varphi_n:\mathbb{D}\to U_n$ such that $\varphi_n(0) = f^n(z_0)$, there exists a unique sequence of inner functions $g_n:\mathbb{D}\to\mathbb{D}$ such that, for all $n\geq 0$,
\begin{enumerate}[(i)]
    \item $g_n$ fixes the origin and
    \item \textcolor{black}{$f|_{U_n}$} lifts to $g_n$, that is, $\varphi_{n+1}\circ g_n = f\circ\varphi_n$.
\end{enumerate}
\end{lemma}
\begin{proof}
The existence, analyticity, and uniqueness of the functions $g_n$ is guaranteed by Lemma \ref{lem:hypdis}. It remains to show that they are inner functions.

Since $\varphi_n(\mathbb{D}) = U_n$ and $U$ is a wandering domain, it is clear that $f\circ\varphi_n(\mathbb{D})$ is not dense in $\Chat$ (indeed, it omits every other $U_m$, $m\neq n + 1$); thus, we can (if necessary) compose $f\circ\varphi_n$ with a M\"obious transformation $m_n$ to obtain a \textit{bounded} map $m_n\circ f\circ \varphi_n:\mathbb{D}\to \Cx$. By Fatou's theorem \cite[~Lemma 6.9]{Hay64}, the radial limit $f\circ\varphi_n(e^{i\theta})$ exists and is finite for $\theta$ outside a set $E_n\subset [0, 2\pi)$ of Lebesgue measure zero. Furthermore, also by Fatou's theorem, there exists a set $E_n'$ of measure zero such that $g_n(e^{i\theta})$ exists for $\theta\in[0, 2\pi)\setminus E_n'$. Now, for every $n\geq 0$, the set $F_n := E_n\cup E_n'$ has measure zero, and for $\theta\in[0,2\pi)\setminus F_n$ we have
\[ \lim_{r\nearrow 1} \varphi_{n+1}\circ g_n(re^{i\theta}) = \lim_{r\nearrow 1} f\circ\varphi_n(re^{i\theta}), \]
where both limits exist and are finite. Since $\varphi_n$, being a covering map, has no asymptotic values in $U_n$ and $f$ is continuous, the right-hand side converges to a point in $\partial U_{n+1}$. Thus, on the left-hand side, we must have $|g_n(re^{i\theta})|\to 1$ as $r\nearrow 1$ by continuity of $\varphi_{n+1}$ and the open mapping theorem. It follows that the $g_n$ are inner functions.
\end{proof}

An immediate consequence of Lemma \ref{lem:lifts} is that, defining $G_n = g_n\circ g_{n-1}\circ \cdots\circ g_0$, we have
\begin{equation} \label{eq:Glifts}
    \text{$\varphi_n\circ G_{n-1} = f^n\circ\varphi_0$ for every $n\geq 1$.}
\end{equation}
However, since the domains $U_n$ are multiply connected, the limiting behaviours of $G_n$ and $(f^n)|_U$ relative to the hyperbolic metric are \textit{not} necessarily the same -- except in one particular case.

\begin{lemma} \label{lem:contr}
With the notation and definitions above, if $G_n(w)\to 0$ as $n\to +\infty$ for all $w\in \mathbb{D}$, then $U$ is a contracting wandering domain.
\end{lemma}
\begin{proof}
Let $w\in\mathbb{D}$; since $G_n(w)\to 0$ as $n\to+\infty$, we know that $d_\mathbb{D}\left(0, G_n(w)\right)\to 0$. By (\ref{eq:Glifts}), we have
\[ d_{U_n}\left(f^n(z_0), f^n\circ\varphi_0(w)\right) = d_{U_n}\left(\varphi_n(0), \varphi_n\circ G_{n-1}(w)\right)\text{ for $w\in\mathbb{D}$}, \]
and applying the Schwarz-Pick lemma to the right-hand side yields
\[ d_{U_n}\left(f^n(z_0), f^n\circ\varphi_0(w)\right) \leq d_\mathbb{D}\left(0, G_{n-1}(w)\right), \]
which goes to zero as claimed.
\end{proof}

For a sequence $(G_n)_{n\in\mathbb{N}}$ as above, \cite[Theorem 2.1]{BEFRS19} tells us that $d_\mathbb{D}\left(0, G_n(w)\right)\to 0$ for $w\in\mathbb{D}$ if and only if
\[ \sum_{n\geq 0} (1 - |g_n'(0)|) = +\infty, \]
or equivalently, if none of the $g_n'(0)$ equal zero, $G_n'(0) \to 0$. Since $f^n$ lifts to $G_n$, the local conformal invariance of hyperbolic distortion (recall (\ref{eq:hypdis})) implies that this is equivalent to saying that a sufficient condition for $U$ to be a contracting wandering domain is
\[ \lim_{n\to +\infty} \left\|Df^n(z_0)\right\|_U^{U_n} = 0. \]

Now, our course of action is to choose a $z_0\in V$, the contracting neighbourhood inside $U$, and show that the condition above holds. To this end, we prove the following hyperbolic version of Landau's theorem. \textcolor{black}{Unlike in Landau's original theorem, the function here is not normalised (if it were, it would be a hyperbolic isometry!), and thus there is dependence on the derivative of $f$.}

\begin{lemma} \label{lem:landau}
Let $f:\Omega_1\to\Omega_2$ be a holomorphic map between hyperbolic domains, let $r > 0$, and let $z_0\in \Omega_1$. If $0 < \|Df(z_0)\|_{\Omega_1}^{\Omega_2} < 1$, then $f\left(B_{\Omega_1}(z_0, r)\right)$ contains a hyperbolic disc of radius
\[ 2\textcolor{black}{L}r^*\|Df(z_0)\|_{\Omega_1}^{\Omega_2}, \]
where $r^* = \tanh(r/2)$ and \textcolor{black}{$L\in (0.5, 0.544)$ is Landau's constant}.
\end{lemma}
\begin{proof}
First, we take uniformising maps $\varphi:\mathbb{D}\to \Omega_1$ and $\psi:\mathbb{D}\to\Omega_2$ with $\varphi(0) = z_0$ and $\psi(0) = f(z_0)$. Then, by Lemma \ref{lem:hypdis}, $f$ admits a lift $F:\mathbb{D}\to\mathbb{D}$ (i.e., $F$ satisfies $\psi\circ F = f\circ\varphi$), which can be chosen to satisfy $F(0) = 0$, and by (\ref{eq:hypdis}) we also have $|F'(0)| = \|Df(z_0)\|_{\Omega_1}^{\Omega_2}$. The ball $B_\mathbb{D}(0, r)$ is contained (by the Schwarz-Pick lemma) in $\varphi^{-1}\left(B_{\Omega_1}(z_0, r)\right)$, and its image under $F$ contains (by Landau's theorem \cite{Lan29}; see \cite{Min82} for a discussion of the bounds given here) a disc of Euclidean radius $Lr^*|F'(0)| = Lr^*\|Df(z_0)\|_{\Omega_1}^{\Omega_2}$, where we have set $r^* = \tanh(r/2)$ so that $d_\mathbb{D}(0, r^*) = r$. Since we do not know where the centre $w_0$ of this disc is, we cannot accurately calculate its hyperbolic size. Nevertheless, we do know that, for any $w$ on its boundary,
\[ d_\mathbb{D}(w_0, w) = \int_\gamma \rho_\mathbb{D}(s)\,|ds|\geq 2\int_\gamma\, |ds| \geq 2|w - w_0| = 2Lr^*\|Df(z_0)\|_{\Omega_1}^{\Omega_2}, \]
where $\gamma\subset\mathbb{D}$ is a hyperbolic geodesic connecting $w_0$ to $w$. Hence, $F\left(B_\mathbb{D}(0, r)\right)$ contains a disc of hyperbolic radius $2Lr^*\|Df(z_0)\|_{\Omega_1}^{\Omega_2}$.

In order to transfer this knowledge from $\mathbb{D}$ to $\Omega_2$, we are going to use tools from \cite{BCMNg04} (see also \cite[Chapter 10]{KL07}). More specifically, for a subdomain $D$ of $\Omega_2$, let $R(D, \Omega_2)$ denote the hyperbolic radius of the largest hyperbolic disc contained in $D$; this is the hyperbolic Bloch constant, and Beardon \textit{et al.} proved in \cite[Lemma 4.2]{BCMNg04} that it is invariant under unbranched covering maps. More specifically, let $B^*$ denote the component of $\psi^{-1}\left(f\left(B_{\Omega_1}(z_0, r)\right)\right)$ containing the origin. Then, since $f\circ \varphi = \psi\circ F$, $F(0) = 0$, and $B^*$ is connected, we have $B^*\supset F(B_\mathbb{D}(0, r))$, and thus
\[ R(f(B_{\Omega_1}(z_0, r)), \Omega_2) = R(B^*, \mathbb{D}) \geq R(F(B_\mathbb{D}(0, r)), \mathbb{D}). \]
Since we know that $F\left(B_\mathbb{D}(0, r)\right)$ contains a disc of hyperbolic radius $2Lr^*\|Df(z_0)\|_{\Omega_1}^{\Omega_2}$, we have
\[ R(F(B_\mathbb{D}(0, r)), \mathbb{D}) \geq 2Lr^*\|Df(z_0)\|_{\Omega_1}^{\Omega_2} \]
and we are done.
\end{proof}

With these results in hand, we can finally prove what we intended. Notice that, by the triangle inequality, the hypotheses of Theorem \ref{thm:mix}(a) imply those of Theorem \ref{thm:contr}.

\begin{theorem} \label{thm:contr}
Let $U$ be a wandering domain of a meromorphic function $f$. If $U$ contains a non-empty open set $V$ such that
\[ d_{U_n}\left(f^n(z), f^n(w)\right)\to 0 \]
for every $z$ and $w$ in $V$, then $U$ is contracting.
\end{theorem}
\begin{proof}
Take any $z_0\in V$, and take uniformising maps $\varphi_n:\mathbb{D}\to U_n$ such that $\varphi_n(0) = f^n(z_0)$. Let $g_n$ be the lifts of $f$ given by Lemma \ref{lem:lifts}; we want to show that
\begin{equation} \label{eq:sum}
    \sum_{n\geq 0} \left(1 - |g_n'(0)|\right) = +\infty,
\end{equation}
whence our theorem will follow by Lemma \ref{lem:contr} and \cite[Theorem 2.1]{BEFRS19}. We can assume that none of the $g_n'(0)$ are zero; indeed, if there are infinitely many such $g_n$, then (\ref{eq:sum}) is trivially true, and if there are only finitely many such $g_n$ we pass to $U_N$, $f^N(z_0)$, and $f^N(V)$ for some sufficiently large $N$. Thus (recall that $G_n = g_n\circ \cdots\circ g_0$), as anticipated on page 8, (\ref{eq:sum}) is equivalent to $|G_n'(0)|\to 0$, or, by (\ref{eq:hypdis}), to $\|Df^n(z_0)\|_U^{U_n}\to 0$.

Assume now that this is not the case; notice that, again by the Schwarz-Pick lemma, the sequence $(\|Df^n(z_0)\|_U^{U_n})_{n\in\mathbb{N}}$ is decreasing, and so if it gets arbitrarily close to zero on a subsequence then it is in fact tending to zero. In other words, there must exist some constant $c > 0$ such that $\|Df^n(z_0)\|_U^{U_n} > c$ for all $n\geq 0$. Choose some $r > 0$ such that $K := \overline{B_U(z_0, r)}\subset V$; then, by Lemma \ref{lem:conv}, we have $\mathrm{diam}_{U_n}\left(f^n(K)\right)\to 0$, while by Lemma \ref{lem:landau} $f^n(K)$ always contains a hyperbolic ball of radius $2Lc\cdot\tanh(r/2)$.
This is clearly a contradiction; we are done.
\end{proof}

\subsection{Proof of Theorem \ref{thm:mix}(b): the semi-contracting case.} \label{ssec:sc}
In this subsection, we want to show that if there exists some point $z_0\in U$ and an open neighbourhood $V$ of $z_0$ such that $u_n(z)\searrow u(z) > 0$ without ever reaching it for $z\in V$ outside of a discrete set ($u_n$ and $u$ as defined at the beginning of Section \ref{sec:u}, with $z_0$ as base point), then the same holds for every $z\in U$ except those for which $f^n(z) = f^n(z_0)$ for some $n\in\mathbb{N}$.

We will divide the proof in two cases: first, we prove that no point $z\in U$ can be contracting relative to $z_0$, and then that no point can be eventually isometric relative to $z_0$.

For the first case, we will revisit the function $u$; in particular, we notice that it is ``$f$-invariant'' in the sense that, if $u^*:U_1\to [0, +\infty)$ is defined taking as a base point $f(z_0)$, then $u^*(f(z)) = u(z)$. By an abuse of notation, we will keep referring to the functions defined on $U_n$ with base points $f^n(z_0)$ as $u$ and hope it will not lead to confusion. Now, what we want to show is that $u(z) = 0$ if and only if $f^n(z) = f^n(z_0)$ for some $n$.

The ``if'' is trivial; for the ``only if'', let $z\textcolor{black}{\in U}$ be such that $f^n(z)\neq f^n(z_0)$ for all $n$, and assume that $u(z) = 0$. Then, since $U$ is semi-contracting in $V$, we can take a closed curve $\gamma\subset V$ that surrounds $z_0$ and avoids the zeros of $u$, which are discrete in $V$ by hypothesis; by Lemma \ref{lem:conv}, $u|_\gamma$ achieves a minimum $c > 0$. Additionally, $V$ is in the Fatou set, meaning that $(f^n)|_V$ is holomorphic for every $n\in\mathbb{N}$, and so by the argument principle $f^n(\gamma)$ surrounds $B_{U_n}(f^n(z_0), c)$ for all $n\in\mathbb{N}$. Therefore, since we have have $d_{U_n}\left(f^n(z), f^n(z_0)\right)\to 0$, there must be $N_1\in\mathbb{N}$ such that $f^{N_1}(\gamma)$ surrounds $f^{N_1}(z)$, and thus -- again by the argument principle -- there exists $w\in V$ surrounded by $\gamma$ such that $f^{N_1}(w) = f^{N_1}(z)$. Now, by $f$-invariance of $u$, we have $u(w) = u(z) = 0$, and so by the definition of $V$ there exists $N_2\in\mathbb{N}$ for which $f^{N_2}(w) = f^{N_2}(z_0)$. Finally, we see that $f^{N_1+N_2}(z) = f^{N_1+N_2}(z_0)$, which is a contradiction since we assumed that $f^n(z)\neq f^n(z_0)$ for all $n\in\mathbb{N}$. It follows that $u(z) > 0$.

Let us now exclude the possibility of eventually isometric points. \textcolor{black}{We show that} if there exists $z\in U$ such that
\begin{equation} \label{eq:evIsom}
\text{$d_{U_n}\left(f^n(z), f^n(z_0)\right) = c(z, z_0) > 0$ for $n\geq N$,}
\end{equation}
then $f:U_n\to U_{n+1}$ is a local hyperbolic isometry (equivalently, an unbranched covering map) for $n\geq N$. Indeed, let $\gamma\subset U_n$ be a hyperbolic geodesic joining $f^n(z)$ to $f^n(z_0)$; then, by the Schwarz-Pick lemma,
\[ \ell_{U_{n+1}}(f\circ\gamma) \leq \ell_{U_n}(\gamma), \]
while by the definition of the hyperbolic distance and (\ref{eq:evIsom}) we have
\[ \ell_{U_{n+1}}(f\circ\gamma) \geq d_{U_{n+1}}\left(f^{n+1}(z), f^{n+1}(z_0)\right) = d_{U_n}\left(f^n(z), f^n(z_0)\right) = \ell_{U_n}(\gamma). \]
Hence, $\ell_{U_{n+1}}(f\circ \gamma) = \ell_{U_n}(\gamma)$, or, equivalently,
\[ \int_{\gamma} \rho_{U_n}(s)\,|ds| = \int_{f\circ\gamma} \rho_{U_{n+1}}(s')\,|ds'|. \]
By a change of variables, this becomes
\[ \int_\gamma \rho_{U_n}(s)\,|ds| = \int_\gamma \rho_{U_{n+1}}\left(f(s)\right)|f'(s)|\,|ds|, \]
and since $0\leq \rho_{U_{n+1}}\left(f(z)\right)|f'(s)|\leq \rho_{U_n}(s)$ by the Schwarz-Pick lemma, we are forced to conclude that $\rho_{U_{n+1}}\left(f(s)\right)|f'(s)| = \rho_{U_n}(s)$ for every $s\in\gamma$. From the equality case of the Schwarz-Pick lemma, we deduce that $f:U_n\to U_{n+1}$ is a local hyperbolic isometry.

From now on, we will exchange $U$ for $U_N$, $z_0$ for $f^N(z_0)$, and $V$ for $f^N(V)$ for some sufficiently large $N$, and work as if $f$ maps one wandering domain onto the next one locally isometrically for every $n\in\mathbb{N}$. We can do that because of the ``$f$-invariance'' of $u$: if $V\subset U$ is semi-contracting relative to $z_0\in V$, then $f^N(V)\subset U_N$ is semi-contracting relative to $f^N(z_0)\in f^N(V)$.

We will see that $V$ being semi-contracting relative to $z_0\in V$ is incompatible with $f:U_n\to U_{n+1}$ being a local hyperbolic isometry. Indeed, let $w$ be any point in $V$ such that $f^n(w)\neq f^n(z_0)$ for all $n$ and $\overline{B_U\left(z_0, d_U(z_0, w)\right)}\subset V$; since $V$ is semi-contracting, $u_n(w) = d_{U_n}\left(f^n(w), f^n(z_0)\right)$ forms a non-increasing sequence that is also not eventually constant, i.e., there is no $N$ such that $u_n(w)$ is constant for $n\geq N$. However, as $f:U_n\to U_{n+1}$ is a local hyperbolic isometry for all $n\geq 0$, we can produce a sequence $(w_n)_{n\in\mathbb{N}}$ in $U\setminus\{w\}$ such that $f^n(w) = f^n(w_n)$ and $d_{U_n}\left(f^n(w), f^n(z_0)\right) = d_U\left(w_n, z_0\right)$ for all $n\in\mathbb{N}$ as follows.

Taking distance-minimising geodesics $\gamma_n\subset U_n$ joining $f^n(z_0)$ to $f^n(w)$, we apply the path lifting property to $f^n:U\to U_n$ (see \cite[Proposition 10]{Don11} or \cite[Proposition 1.30]{Hat02}) to obtain a curve $\tilde{\gamma}_n\subset U$ joining $z_0$ to a point $w_n\in U$ such that $f^n(w_n) = f^n(w)$. Since $f^n:U\to U_n$ is a local hyperbolic isometry, it preserves curve length, and thus
\[ d_{U_n}\left(f^n(z_0), f^n(w)\right) = \ell_{U_n}(\gamma_n) = \ell_U(\tilde{\gamma}_n) \geq d_U(z_0, w); \]
on the other hand, by the Schwarz-Pick lemma,
\[ d_U(z_0, w_n) \geq d_{U_n}\left(f^n(z_0), f^n(w)\right), \]
and therefore $d_U(z_0, w_n) = d_{U_n}\left(f^n(z_0), f^n(w)\right)$ for all $n\in\mathbb{N}$. Notice, though, that $\left(d_{U_n}\left(f^n(z_0), f^n(w)\right)\right)_{n\in\mathbb{N}}$ is -- by hypothesis -- a non-constant non-increasing sequence, so that $w_n\neq w$ for all large enough $n$.

By the same token, the sequence $(w_n)_{n\in\mathbb{N}}$ is confined to the annulus $\{z\in U : u(w) \leq d_U(z, z_0) \leq d_U(w, z_0)\}\subset V$, and hence has an accumulation point $w^*\in V$. Since $u$ is $f$-invariant, every $w_n$ satisfies $u(w_n) = u(w)$, and thus by the continuity of $u$ we have $u(w^*) = u(w)$. But it is also the case by the definition of the points $w_n$ that $u(w^*) = u(w) = \lim_{n\to+\infty} d_U(w_n, z_0)$, and so by continuity of the hyperbolic distance we have
\[ u(w^*) = d_U(w^*, z_0). \]
It follows from the definition of $u$ that $w^*$ is an eventually isometric point relative to $z_0$ lying in $V$, which is a contradiction.

Together, these two arguments show that every point in $U$ is semi-contracting relative to $z_0$.

\subsection{Proof of Theorem \ref{thm:mix}(c): The eventually isometric case.}
Here, we assume that $U$ contains a non-empty open set $V$ and a point $z_0\in V$ such that $d_{U_n}\left(f^n(z), f^n(z_0)\right) = c(z, z_0) > 0$ for every $z\in V$ (except at most countably many points for which $f^k(z) = f^k(z_0)$ for some $k\in\mathbb{N}$) and all sufficiently large $n$ (say, $n\geq N$, with $N$ locally uniform over compact subsets of $V$). We want to show that the same holds for every $w\in U$ relative to $z_0\in V$. Perhaps surprisingly, this is the most delicate case we will deal with here -- and it needs certain machinery that we take the time to introduce now.

We consider the set $\mathcal{H}_2$ of hyperbolic surfaces, identifying any two surfaces that are isometric. In other words, $\mathcal{H}_2$ is a space of equivalence classes of hyperbolic surfaces. We can ``refine'' it a little by considering \textit{marked} hyperbolic surfaces: pairs $(S, p)$ where $S\in\mathcal{H}_2$ and $p\in S$ is a base point. The space of all such pairs, again up to isometry equivalence, is denoted $\mathcal{H}_2^*$; we will now imbue it with a topology.

\begin{definition}
Let $\left((S_n, p_n)\right)_{n\in\mathbb{N}}$ be a sequence of marked hyperbolic surfaces. We say that the sequence converges to the marked surface $(S^*, p^*)\in\mathcal{H}_2^*$ \textit{in the sense of Gromov} if, for every $r > 0$, there exists a sequence of smooth orientation-preserving diffeomorphisms $\phi_n:U_n\to \phi_n(U_n)\subset S_n$ such that the following hold.
\begin{enumerate}[(i)]
    \item Each $U_n\subset S^*$ is a neighbourhood of $p^*$ containing $B_{S^*}(p^*, r)$.
    \item For all $n\in\mathbb{N}$, $\phi_n(p^*) = p_n$.
    \item Each $\phi_n$ is $K_n$-bi-Lipschitz relative to the hyperbolic metrics of $S^*$ and $S_n$, and $K_n\to 1$ as $n\to +\infty$.
\end{enumerate}
\end{definition}

Despite its (relatively) intuitive definition, this topology -- called the \textit{geometric topology} --  offers little insight into its own properties. Fortunately, there is an equivalent way of defining it via Kleinian groups, originally due to Claude Chabauty, which we will not state here. The following result was proved using this alternative definition (see \cite[Theorem E.1.10]{BP92}, \cite[Corollary I.3.1.7]{CEM06}, or \cite[Proposition 7.8]{MT98}).

\begin{lemma} \label{lem:geom}
Let $\left((S_n, p_n)\right)_{n\in\mathbb{N}}$ be a sequence in $\mathcal{H}_2^*$ such that the injectivity radius of $S_n$ at $p_n$ is at least some $r > 0$ for all $n\in\mathbb{N}$. Then, there exists a subsequence $\left((S_{n_k}, p_{n_k})\right)_{k\in\mathbb{N}}$ converging in the sense of Gromov to a marked hyperbolic surface $(S^*, p^*)$.
\end{lemma}

Notice that Lemma \ref{lem:geom} makes no claim about the limit surface (or, to be more precise, the equivalence class of limit surfaces). The nature of the limit surface can, in fact, be extremely counter-intuitive -- but since appreciating the intricacies of the geometric topology is not our aim here, Lemma \ref{lem:geom} will suffice.

Convergence in the sense of Gromov can be though of as a ``locally uniform convergence'' of the surfaces' geometry around the base point. Thus, Lemma \ref{lem:geom} is, in a certain sense, a ``normal families'' criterion for $\mathcal{H}_2$. We will also need more conventional results on normal families, such as the following restatement of the Arzel\`a-Ascoli theorem for Lipschitz functions; see \cite[Theorem~7.1]{BM14}.
\begin{lemma} \label{lem:normal}
Let $X$ and $Y$ be Riemann surfaces with complete conformal metrics $d_X$ and $d_Y$, and suppose that $\{f_\alpha : X\to Y\}_{\alpha\in A}$ is a family of locally uniformly $d_X$-to-$d_Y$ Lipschitz functions. Then, $\{f_\alpha\}_{\alpha\in A}$ is a normal family if and only if there exists $x\in X$ such that $\{f_\alpha(x) : \alpha\in A\}$ is relatively compact in $Y$.
\end{lemma}

With this in mind, let us begin our study of the eventually isometric case in Theorem \ref{thm:mix}.

Let $U$, $f$, $V\subset U$ and $z_0\in V$ be as in the statement of Theorem \ref{thm:mix}(c). We can exclude the possibility of contracting points in $U$ (relative to $z_0$) by an argument similar to the one used in Subsection \ref{ssec:sc}: the iterates of such a point $w\in U$ would eventually intersect $f^n(V)$, and the $f$-invariance of $u$ would imply that $u(w) > 0$, contradicting our choice of $w$.

Now, suppose that there exists a point $w\in U\setminus V$ that is semi-contracting relative to $z_0$. The same argument as in Subsection \ref{ssec:sc} shows that $f:U_n\to U_{n+1}$ is a local hyperbolic isometry for $n\geq N$ ($N$ here being chosen according to some compact subset of $V$), but the rest of the argument does not carry through -- there is nothing unexpected about finding an eventually isometric point relative to $z_0$ in $V$; think, for instance, of the ``annulus model'' described in \cite[Section 2]{Fer21}. Instead, we will proceed with a ``normal families'' argument. As before, we swap $U$ for $U_N$, $z_0$ for $f^N(z_0)$, and $V$ for $f^N(V)$ while keeping the same notation.

First, we take a universal covering map $\varphi_0:\mathbb{D}\to U$ with $\varphi_0(0) = z_0$, and build the functions $\psi_n:\mathbb{D}\to U_n$ given by $\psi_n(z) = f^n\circ\varphi_0(z)$; since $f^n$ is a local hyperbolic isometry, these are all universal covering maps with $\psi_n(0) = f^n(z_0)$. Furthermore, the hypothesis that $V$ is locally eventually isometric implies that, if $r_0$ is such that $\overline{B_U(z_0, r_0)}\subset V$ is an embedded disc (and $N$ was chosen accordingly), then $f^n$ maps $B_U(z_0, r_0)$ isometrically onto $f^n(B_U(z_0, r_0))$ for all $n$. In particular the sets $f^n\left(B_U(z_0, r_0)\right)$ will all be embedded discs of hyperbolic radius $r_0$. Therefore, the injectivity radii at $f^n(z_0)$ of the hyperbolic surfaces $U_n$ are uniformly bounded below by $r_0$. Thus, by Lemma \ref{lem:geom}, the sequence of marked hyperbolic surfaces $((U_n, f^n(z_0)))_{n\in\mathbb{N}}$ admits a subsequence $((U_{n_k}, f^{n_k}(z_0)))_{k\in\mathbb{N}}$ converging in the sense of Gromov to a marked hyperbolic surface $(U^*, z^*)$ (say).

It follows from the definition, then, that we can take some $r > 10d_U(z_0, w)$ (where $w\in U\setminus V$ is assumed to be semi-contracting relative to $z_0$) and find diffeomorphisms $\phi_k:\overline{B_{U^*}(z^*, r)}\to \phi_k\left(\overline{B_{U^*}(z^*, r)}\right)\subset U_{n_k}$ such that $\phi_k(z^*) = f^{n_k}(z_0)$ and each $\phi_k$ is $K_k$-bi-Lipschitz relative to the hyperbolic metrics on $U^*$ and $U_{n_k}$, with $K_k\to 1$ as $k\to +\infty$. Since $K_k\to 1$, we can assume without loss of generality that $\sup_k K_k < 10$, so that every $\phi_k^{-1}$ is defined on a hyperbolic ball of radius $d_U(z_0, w)$ around $f^{n_k}(z_0)$. In particular, the functions
\[ \widehat{\psi}_k = \phi_k^{-1}\circ \psi_{n_k} : B_\mathbb{D}(0, d_U(z_0, w)) \to U^* \]
are all well-defined and map $0$ to $z^*$. We want to show that $\{\widehat{\psi}_k\}_{k\in\mathbb{N}}$ is a normal family in the sense of precompactness relative to locally uniform convergence (see \cite{BM14}, as well as \cite[Theorem 1.21]{BF14} and \cite{HM20} for an account of this interpretation of normality) with analytic limit functions; to this end, we make two claims.

\begin{claim} \label{cl:qr}
Each $\widehat{\psi}_k$ is $(K_k)^2$-quasiregular.
\end{claim}
\begin{proof}
Being $K_k$-bi-Lipschitz, each $\phi_k$ is $(K_k)^2$-quasiconformal (see, for instance, \cite[Chapter 1]{BGMR13} or \cite[Proposition 4.5.14]{Hub06}), and the composition of a restriction of a universal covering map (which is, of course, $1$-quasiregular) with a $(K_k)^2$-quasiconformal map yields a $K_k^2$-quasiregular map.
\end{proof}

\begin{claim}
The family $\{\widehat{\psi}_k\}_k$ is uniformly Lipschitz relative to the hyperbolic metric.
\end{claim}
\begin{proof}
For $z$ and $z'$ in $B_w = B_\mathbb{D}(0, d_U(z_0, w))$, we have, using the bi-Lipschitz constant of $\phi_k$ and the Schwarz-Pick lemma respectively,
\[ d_{U^*}\left(\phi_k^{-1}\circ\psi_{n_k}(z), \phi_k^{-1}\circ\psi_{n_k}(z')\right) \leq K_kd_{U_{n_k}}\left(\psi_{n_k}(z), \psi_{n_k}(z')\right) \leq K_kd_\mathbb{D}(z, z') \leq K_kd_{B_w}(z, z'). \]
Since $K_k\to 1$, we have that $K := \sup_k K_k < 10$ is a uniform Lipschitz constant for $\phi_k^{-1}\circ \psi_{n_k} = \widehat{\psi}_k$.
\end{proof}

Now, since $\{\widehat{\psi}_k(0) : k\in\mathbb{N}\} = \{z^*\}$ is clearly a relatively compact subset of $U^*$, Lemma \ref{lem:normal} tells us that $(\widehat{\psi}_k)_k$ is a normal family, and so admits a subsequence $(\widehat{\psi}_{k_m})_m$ converging locally uniformly to a limit function
\[ \psi^*:B_\mathbb{D}(0, d_U(z_0, w))\to U^*. \]
Since $K_k\to 1$ and the limit of $K$-quasiregular functions is $K$-quasiregular, it follows by Claim \ref{cl:qr} that $\psi^*$ is $1$-quasiregular (see for example \cite[Theorem 4.2]{BGMR13}, \cite{HM20}, \cite[Corollary 5.5.7]{AIM09}, or \cite[Theorem VI.8.6]{Ric93}) and hence analytic by the quasiregular version of Weyl's lemma (see \cite[Proposition 1.37]{BF14}).

We want to show that $\psi^*$ is not constant; to that end, take a point $z'\in B_\mathbb{D}(0, d_U(z_0, w))$ such that $\varphi_0(z')\in V\setminus\{z_0\}$. Then, by the definition of the maps $\widehat{\psi}_n$ and the bi-Lipschitz property of $\phi_k$,
\[ d_{U^*}\left(z^*, \widehat{\psi}_{k_m}(z')\right)\geq \frac{d_{U_n}\left(\psi_{k_m}(0), \psi_{k_m}(z')\right)}{K_{k_m}} \geq \frac{d_{U_n}\left(\psi_{k_m}(0), \psi_{k_m}(z')\right)}{10}; \]
since $\psi_n = f^n\circ\varphi_0$ and $f^n$ is isometric when restricted to $V$, we have that $d_{U_n}(\psi_{k_m}(0), \psi_{k_m}(z')) = d_U(z_0, \varphi_0(z')) > 0$ for all $m$. Thus, by making $m\to +\infty$, we see that $d_{U^*}(\psi^*(0), \psi^*(z')) > 0$, and so $\psi^*$ is indeed non-constant.

We are now in position to make a case against the existence of $w$, the semi-contracting point relative to $z_0$. The fact that $f:U_n\to U_{n+1}$ are local hyperbolic isometries implies that, in order for $d_{U_n}\left(f^n(w), f^n(z_0)\right)$ to decrease infinitely many times, we must (as in Subsection \ref{ssec:sc}) be able to find a sequence $w_n\in U\setminus\{w\}$ such that $f^n(w_k) = f^n(w)$ for $1\leq k\leq n$ and $d_U(w_n, z_0) = d_{U_n}\left(f^n(w), f^n(z_0)\right) < d_U(z_0, w)$. By the completeness of the hyperbolic metric in $U$, the sequence $w_n$ admits an accumulation point $w^*\in B_U\left(z_0, d_U(z_0, w)\right)$, and this lifts to a sequence $\tilde{w}_n\in B_\mathbb{D}\left(0, d_U(z_0, w)\right)$ with an accumulation point $\tilde{w}^*\in B_\mathbb{D}\left(0, d_U(z_0, w)\right)$. However, repeated application of the triangle inequality yields
\begin{multline*}
d_{U^*}\left(\psi^*(\tilde{w}_{k_m}), \psi^*(\tilde{w}^*)\right) \leq \\
d_{U^*}\left(\psi^*(\tilde{w}_{k_m}), \widehat{\psi}_{k_l}(\tilde{w}_{k_m})\right) + d_{U^*}\left(\widehat{\psi}_{k_l}(\tilde{w}_{k_m}), \widehat{\psi}_{k_l}(\tilde{w}_{k_l})\right) + d_{U^*}\left(\widehat{\psi}_{k_l}(\tilde{w}_{k_l}), \psi^*(\tilde{w}^*)\right),
\end{multline*}
where $m\in\mathbb{N}$ is fixed and $l\geq m$. By the definitions of $\tilde{w}_n$ and $\widehat{\psi}_k$, we have that $\widehat{\psi}_{k_l}(\tilde{w}_{k_l}) = \widehat{\psi}_{k_l}(\tilde{w}_{k_m}$), so that the middle term in the sum above vanishes. As for the other two, taking the limit $l\to +\infty$ makes them arbitrarily small, since $\widehat{\psi}_{k_l}$ converges to $\psi^*$ locally uniformly. It follows that $\psi^*(\tilde{w}_{k_m}) = \psi^*(\tilde{w}^*)$ for all $m\in\mathbb{N}$: a contradiction, since $\psi^*$ is analytic and as such $(\psi^*)^{-1}(p)$ is discrete for any $p\in U^*$.

We have completed the proof of case (c), and hence finished the proof of Theorem~\ref{thm:mix}.

\begin{remark} This argument can be adapted to rule out the existence of contracting points \textit{and} to drop the dependence on the base point $z_0$, which takes care of the eventually isometric case in its general form. Indeed, the contradiction above can be obtained for any pair of points $z, w\in U$ such that $d_{U_n}(f^n(z), f^n(w))$ is not eventually constant -- one need only choose a sufficiently large radius for the discs where the $\phi_k$ are to be defined.
\end{remark}

\section{A semi-contracting multiply connected wandering domain} \label{sec:example}
\subsection{The patient.} \label{ssec:patient}
Benini \textit{et al.} constructed the first known examples of semi-contracting wandering domains using approximation theory (the one we are interested in is \cite[Example 2(a)]{BEFRS19}). As such, we can describe its properties mostly in an asymptotic fashion -- but by controlling the rate of convergence, this will suffice for our ends.

To describe the example, we shall need the following sets and functions. Let $T_n:\Cx\to\Cx$ and $b_n:\mathbb{D}\to\mathbb{D}$ be defined as
\[ \text{$T_n(z) = z + 4n$ and $b_n(z) = z\cdot\frac{z + a_n}{1 + a_nz}$,} \]
where $a_n$ is a real sequence satisfying $0 < a_n < 1$ and $a_n\nearrow 1$ fast enough that
\[ \lambda := \prod_{n = 1}^{+\infty} a_n > 0; \]
as our construction progresses, we will impose further restrictions on the speed with which $a_n\nearrow 1$. Define also the discs $\Delta_n' := B(4n, r_n)$ and $\Delta_n := B(4n, R_n)$, where $0 < r_n < 1 < R_n$, both sequences $(r_n)_{n\in\mathbb{N}}$ and $(R_n)_{n\in\mathbb{N}}$ tend to $1$ as $n\to +\infty$, and the rate of convergence is relatively ``free'' -- that is, $1 - r_n$ and $R_n - 1$ have upper bounds depending only on $1 - a_n$.

With all that in place, their example consists of an entire function $f$ with a bounded simply connected wandering domain $U$ such that, for all $n\geq 0$,
\begin{enumerate}[(A)]
    \item $\overline{\Delta_n'}\subset U_n\subset \Delta_n$, so, in particular, $U_n$ ``is asymptotically a disc'' (recall that $U_n$ is the Fatou component containing $f^n(U)$);
    \item $\left|f(z) - T_{n+1}\circ b_{n+1}\circ T_n^{-1}(z)\right| < \epsilon_{n+1}$ for $z\in \overline{\Delta_n}$, where $\epsilon_n < (R_n - r_n)/4$;
    \item $f^n(0) = 4n$;
    \item $\deg f|_{U_n} = \deg b_{n+1} = 2$ and $f|_{U_n}$ satisfies
    \[ \text{$|f(z) - z - 4|\to 0$ locally uniformly for $z\in U_n$ as $n\to +\infty$.} \]
    In particular, the critical points $z_n$ of $f|_{U_n}$ satisfy
    \[ \text{$\dist(z_n, \partial U_n)\to 0$ as $n\to+\infty$,} \]
    where $\dist$ denotes Euclidean distance.
\end{enumerate}

It was shown in \cite[Example 2]{BEFRS19} that the wandering domain $U$ is semi-contracting; let us take a closer look at it.

Let $B_n(z) := b_n\circ b_{n-1}\circ \cdots \circ b_1(z)$ for $n\in\mathbb{N}$. By Montel's theorem, $B_n$ admits a subsequence converging locally uniformly to some $B:\mathbb{D}\to\overline{\mathbb{D}}$; since $b_n(0) = 0$ for all $n\in\mathbb{N}$, we have $B(0) = 0$, and thus $B:\mathbb{D}\to\mathbb{D}$. Furthermore, by the Weiestrass convergence theorem and the assumptions on $a_n$,
\[ B'(0) = \lim_{n\to+\infty} B_n'(0) = \prod_{n=1}^{+\infty} a_n = \lambda > 0, \]
and so $B$ is non-constant. Additionally, since $B_n'(0) = \prod_{k=1}^n a_k > 0$ and, by Schwarz's lemma applied to the sequence $(B_n)_{n\in\mathbb{N}}$,
\begin{equation} \label{eq:Bndec}
\text{$|B_n(z)|\searrow |B(z)|$ for all $z\in\mathbb{D}$,}
\end{equation}
the limit function $B$ is unique.

Next, \color{black}we recall \cite[Corollary 2.4]{BEFRS19}:
\begin{lemma} \label{lem:cor2.4}
Let $h:\mathbb{D}\to\mathbb{D}$ be a holomorphic function with $h(0) = 0$ and $|h'(0)| = \mu$. Then, for all $w\in\mathbb{D}$,
\[ h_1(|w|) := |w|\cdot\frac{\mu - |w|}{1 - \mu|w|} \leq |h(w)| \leq |w|\cdot\frac{\mu + |w|}{1 + \mu|w|} =: h_2(|w|). \]
\end{lemma}

Elementary calculus shows that the function $h_1:[0, 1]\to\mathbb{R}$ is a concave function with maximum
\[ \frac{2\mu^2 + (2 - \mu^2)\sqrt{1 - \mu^2} - 2}{\mu^2\sqrt{1 - \mu^2}}. \]
Thus, if we \color{black}choose $c\in(0, 1)$ \textcolor{black}{such that
\[ c < \frac{2\lambda^2 + (2 - \lambda^2)\sqrt{1 - \lambda^2} - 2}{\lambda^2\sqrt{1 - \lambda^2}}, \]
then there exists a round annulus $A\subset \mathbb{D}$ centred at $0$ such that $|B(z)| > c$ for $z\in A$. Recall now (\ref{eq:Bndec}); in particular, since $\left(|B_n(z)|\right)_{n\in\mathbb{N}}$ is a non-increasing sequence for all $z\in\mathbb{D}$ and $B_n(z) = b_n\circ b_{n-1}\circ \cdots\circ b_1(z)$, we also have $|b_n(z)| > c$ for $z\in A$ for all $n\in\mathbb{N}$.}

\textcolor{black}{Take now some positive $c' < c$; applying Rouche's theorem to the condition (B) satisfied by $f$ we conclude that, for all sufficiently large $n$, there exists a topological annulus $A_n'\subset \overline{\Delta_n'}$ such that $|f(z) - 4(n+1)| > c'$ for all $z\in A_n'$, and $A_n'$ surrounds the disc $\{z\in\Cx : |z - 4n| \leq c'\}$.}

\subsection{The surgery.}
At each $U_n$, we want to cut out a small disc and replace $f$ by an appropriately re-scaled \textcolor{black}{version of the Joukowski map $z\mapsto z + z^{-1}$} inside of it, giving us a pole in each domain. After that, we must join $f$ to the Joukowski map through quasiconformal interpolation in an appropriate annulus \textcolor{black}{$A_n\subset U_n$ such that $\overline{A_n}\subset U_n$}. Since we want the resulting quasiregular map $g_0$ to be quasiconformally conjugate to a meromorphic one, there are two \textcolor{black}{conditions we require to hold}:
\begin{enumerate}
    \item Ensure that \textcolor{black}{the dilatations $K_n$ of the interpolating map in $A_n$ (respectively) satisfy}
    \[ \prod_{n=1}^{+\infty} K_n < +\infty; \]
    \item \textcolor{black}{For any $z\in A_n$, its orbit under $g_0$ does not intersect $A_m$ for $m < n$.}
\end{enumerate}

The two main results we will use in this surgery are both due to Kisaka and Shishikura; first, we have a way to interpolate quasiconformally in \textcolor{black}{round} annuli \cite[Lemma 6.2]{KS08}.

\begin{lemma} \label{lem:interpolate}
Let $k\in\mathbb{N}$, $0 < R_1 < R_2$ and $\varphi_j$ be analytic on a neighbourhood of $C_j := \{|z| = R_j\}$ such that $\varphi_j|_{C_j}$ winds around the origin $k$ times ($j = 1, 2$). If there exist positive constants $\delta_0$ and $\delta_1$ such that
\begin{equation} \label{eq:cond1}
    \left|\log\left(\frac{\varphi_2(R_2e^{i\theta})}{R_2^k}\frac{R_1^k}{\varphi_1(R_1e^{i\theta})}\right)\right|\leq \delta_0
\end{equation}
and
\begin{equation} \label{eq:cond2}
    \left|z\frac{d}{dz}\left(\log\frac{\varphi_j(z)}{z^k}\right)\right| \leq \delta_1,\quad z = R_je^{i\theta},\quad (j = 1, 2)
\end{equation}
for every $\theta\in [0, 2\pi)$, and if $\delta_0$ and $\delta_1$ satisfy
\begin{equation} \label{eq:KSc}
    C = 1 - \frac{1}{k}\left(\frac{\delta_0}{\log\left(R_2/R_1\right)} + \delta_1\right) > 0,
\end{equation}
then there exists a quasiregular map
\[ H : \{z : R_1 \leq |z|\leq R_2\}\to\Cx^* \]
without critical points that interpolates between $\varphi_1$ and $\varphi_2$ and has dilatation constant
\begin{equation} \label{eq:KSk}
    K_H \leq \frac{1}{C}.
\end{equation}
\end{lemma}

Second, \textcolor{black}{we need} a sufficient condition for us to conjugate the resulting quasiregular map to a meromorphic one \cite[Theorem 3.1]{KS08}. Although it was originally stated only for entire maps, it is easy to see how to deal with the presence of poles\textcolor{black}{: this lemma is a particular case of the necessary and sufficient conditions given by Sullivan's Straightening Theorem, which is flexible enough for transcendental meromorphic maps (see \cite[Sections 5.2 and 5.3]{BF14}).}

\begin{lemma} \label{lem:straighten}
Let $g:\Cx\to\Chat$ be a quasiregular map. Suppose there are (disjoint) measurable sets $E_j\subset \Cx$, $j = 1, 2, \ldots$, such that
\begin{enumerate}[(i)]
    \item for almost every $z\in\Cx$, the $g$-orbit of $z$ meets $E_j$ at most once for every $j$,
    \item $g$ is $K_j$-quasiregular on $E_j$,
    \item $K_\infty := \prod_{j\geq 1} K_j < +\infty$, and
    \item $g$ is holomorphic (Lebesgue) a.e. outside $\bigcup_{j\geq 1} E_j$.
\end{enumerate}
Then, there exists a $K_\infty$-quasiconformal map $\phi$ (``fixing infinity'') such that $f = \phi\circ g\circ \phi^{-1}$ is a meromorphic function.
\end{lemma}

We have our patient and our tools; let us begin the surgery. We refer to Figure \ref{fig:surgery} for a sketch of some of the sets and functions here and their relations to each other. First, we take the points $z_n := 4n\in U_n$, $n\geq N$, which form an orbit under $f$\textcolor{black}{; $N$ here is large enough that the annuli $A_n'$ described at the end of Subsection \ref{ssec:patient} exist for $n\geq N$}. We take some $r > 0$ such that, for $n\geq N$, the circle $C_n = \{z : |z_n - z| = r\}$ (shown in blue in Figure \ref{fig:surgery}) is surrounded by \textcolor{black}{$A_n'$}. Inside the discs $\{z : |z_n - z| \leq r\}$, we will remove $f$ and transplant appropriately translated versions of the following re-scaled Joukowski map:
\[ J_n(z) = \frac{\mu_n r^2}{\mu_n^2r^2 - 1}\left(\mu_nz + \frac{1}{\mu_nz}\right), \]
where the $\mu_n > 1/r$ are parameters that will give us great control over the dilatation of the interpolating maps. The strange scaling constant here requires explaining; its role is to guarantee that $J_n(C_n)$ surrounds $C_{n+1}$, \textcolor{black}{which will help us enforce condition 2}. It can be calculated as follows: as explained in \cite[pp. 94--95]{Ahl79}, the function $z\mapsto z + 1/z$ maps a circle of radius $\rho > 1$ injectively onto an ellipse with major semi-axis $\rho + 1/\rho$ and minor semi-axis $\rho - 1/\rho$. Denoting the (as yet unknown) scaling constant by $\lambda$, the condition ``$J_n(C_n)$ surrounds $C_{n+1}$'' then becomes
\[ \lambda\left(\mu_nr - \frac{1}{\mu_n r}\right) = r, \]
and solving for $\lambda$ yields precisely
\[ \lambda = \frac{\mu_nr^2}{\mu_n^2r^2 - 1}. \]

Next, since $f|_{U_n}$ converges locally uniformly to the translation $z\mapsto z + 4$ (in the sense outlined in property (D)), we can take some $r' > r$ with $C_n' := \{z : |z_n - z| = r'\}$ (red in Figure \ref{fig:surgery}) also surrounded by $A_n'$, and such that $|f(C_n') - z_{n+1}| > r + \delta$ for all \textcolor{black}{sufficiently} large $n$. Now, if $n$ is (again) large enough, the only pre-image of $z_{n+1}$ under $f$ surrounded by $C_n'$ is $z_n$ itself, and hence $\ind(f\circ C_n', z_{n+1}) = 1$ (the notation $\ind(\gamma, \alpha)$ denotes the winding number of the curve $\gamma$ around the point $\alpha\in\Cx$). We are ready to start interpolating the maps $T_{n+1}\circ J_n\circ T_n^{-1}$ on $C_n$ and $f$ on $C_n'$. For the sake of convenience, we pass to the unit disc, interpolating instead the functions $J_n$ and $T_{n+1}^{-1}\circ f\circ T_n$ on $\{z\in\Cx : |z| = r\}$ and $\{z\in\Cx : |z| = r'\}$ (respectively).

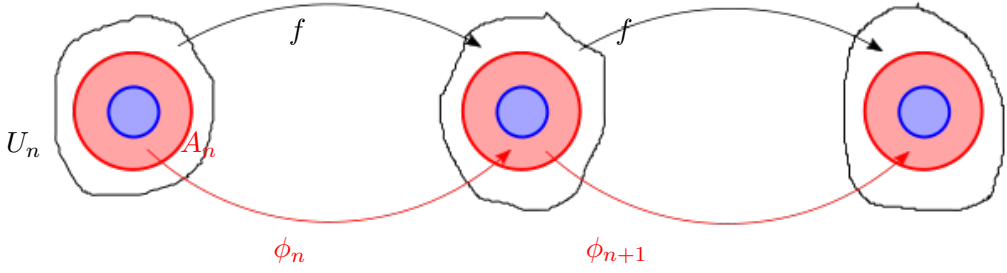
\begin{figure}[!h]
    \centering
    \begin{tikzpicture}
        \node[anchor=south west,inner sep=0] at (0,0)
        {\includegraphics[width=0.9\columnwidth]{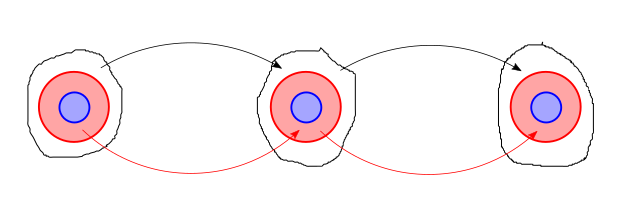}};

        \node at (3.8,3.2) {$f$};
        \node at (8.1,3.2) {$f$};

        \node[red] at (3.7,0.3) {$\phi_n$};
        \node[red] at (8,0.3) {$\phi_{n+1}$};

        \node[red] at (2.5,1.7) {$A_n$};
        \node at (0.2,1.7) {$U_n$};
    \end{tikzpicture}
    \caption{In the blue discs bounded by $C_n$, the original map $f$ was replaced by appropriately translated versions of $\gamma_n$. To connect this with $f$, we interpolate on the red annuli \textcolor{black}{$A_n$} using the quasiconformal maps $\phi_n$.}
    \label{fig:surgery}
\end{figure}

Needless to say, we're applying Lemma \ref{lem:interpolate} with $k = 1$; first, we'd like to draw attention to condition (\ref{eq:cond2}) in the case $j = 1$, i.e., for the function $J_n$. The left-hand quantity takes the form
\[ \left|\frac{2}{\mu_n^2r^2e^{2i\theta} + 1}\right|, \]
which has the upper bound (achieved for $\theta = \pi/2$)
\[ \frac{2}{\mu_n^2r^2 - 1}. \]
Clearly, taking $\mu_n\to +\infty$ gives us excellent control over how small this term is; it is through this ``trick'' that we will control the dilatation introduced by the Joukowski map.

For the case $j = 2$, we have a function of the form $f(z) = b_{n+1}(z) + \epsilon_n(z)$, where $\epsilon_n\to 0$ uniformly as $n\to +\infty$ and $b_n$ are the Blaschke products described in Subsection \ref{ssec:patient}. The exact form of \textcolor{black}{the left-hand side of} condition (\ref{eq:cond2}) is
\begin{equation} \label{eq:lh1}
    \left|z\frac{d}{dz}\left(\log\frac{b_{n+1}(z) +\epsilon_n(z)}{z}\right)\right|,
\end{equation}
\textcolor{black}{and by controlling how fast $a_n\nearrow 1$ and applying Cauchy's integral formula to $\epsilon_n$, we can ensure that $\epsilon_n$ and its derivative tend to zero as fast as necessary to \textcolor{black}{control the size of the constants in Lemma \ref{lem:interpolate}}}. Thus, up to a small error term \textcolor{black}{tending to zero} arbitrarily fast, \textcolor{black}{(\ref{eq:lh1})} takes the form
\[ \left|z\frac{b_{n+1}'(z)}{b_{n+1}(z)} - 1\right| = \left|\frac{(1 +a_{n+1})z +a_{n+1}^2z^2 +a_{n+1}}{(1 +a_{n+1}^2)z + a_{n+1}z^2 + a_{n+1}} - 1\right|, \]
which again \textcolor{black}{tends to zero} arbitrarily fast by choosing an appropriate sequence $a_n\nearrow 1$. Finally, \textcolor{black}{the left-hand side of} condition (\ref{eq:cond1}) takes the form
\[ \left|\log\left(\frac{b_{n+1}(r'e^{i\theta}) + \epsilon_n(r'e^{i\theta})}{r'e^{i\theta}}\frac{re^{i\theta}}{J_n(re^{i\theta})}\right)\right| \]
for $\theta\in[0, 2\pi)$, and we can use the triangle inequality to \textcolor{black}{bound it by}
\[ \left|\log\frac{b_{n+1}(r'e^{i\theta}) + \epsilon_n(r'e^{i\theta})}{r'e^{i\theta}}\right| +\left|\log\frac{re^{i\theta}}{J_n(re^{i\theta})}\right|. \]
Ignoring the error term again, this becomes
\begin{equation} \label{eq:lh2}
    \left|\log\frac{r'e^{i\theta} + a_{n+1}}{1 + a_{n+1}r'e^{i\theta}}\right| + \left|\log\frac{re^{i\theta}}{\frac{\mu_n^2r^3e^{i\theta}}{\mu_n^2r^2 - 1} - \frac{r}{e^{i\theta}(\mu_n^2r^2 - 1)}}\right|.
\end{equation}
Thus, we see that, by making appropriate, independent choices of $a_n\nearrow 1$ and $\mu_n\to +\infty$, we can make \textcolor{black}{(\ref{eq:lh2})} go to zero arbitrarily fast \textcolor{black}{as $n\to +\infty$}.

It follows that we can arrange conditions (\ref{eq:cond1}) and (\ref{eq:cond2}) to hold in each annulus $A_n := \{z\in \Cx : r < |z - z_n| < r'\}$ with constants $\delta_{n,0}$ and $\delta_{n,1}$ that are as small as we need them to be. Thus, by invoking Lemma \ref{lem:interpolate}, we obtain a sequence of quasiconformal maps $\phi_n$ that interpolate between $J_n$ and $f$ on the annuli $A_n$ \textcolor{black}{with dilatation as close to one as we please, say $K_n < 1 + 1/n^2$, by (\ref{eq:KSc}) and (\ref{eq:KSk}). It follows that the $K_n$ can be made to satisfy
\[ K_\infty := \prod_{n=1}^{+\infty} K_n < +\infty. \]}

\textcolor{black}{We define now the map
\[ g_0(z) := \begin{cases}
            J_n(z) & z\in \mathrm{int}(C_n), n\geq N, \\
            \phi_n(z) & z\in \overline{A_n}, n\geq N, \\
            f(z) & \text{elsewhere}, \end{cases} \]
and claim that it satisfies the hypotheses of Lemma \ref{lem:straighten}. Indeed, for any $z\in A_n$, the fact that $J_n(C_n)$ surrounds $C_{n+1}$ while $C_{n+1}'$ surrounds $f(C_n')$ imply that $g_0(z)\in A_{n+1}$, and so the $g_0$-orbit of every $z\in \Cx$ meets each $A_n$ at most once, guaranteeing hypothesis (i). The hypotheses (ii) to (iv) are also satisfied by the construction of the $J_n$ and our choices of $a_n$ and $\mu_n$, and so we \textcolor{black}{can} apply Lemma \ref{lem:straighten} and obtain a $K_{\infty}$-quasiconformal map $\psi$ such that
\[ g(z) := \psi\circ g_0\circ\psi^{-1}(z) \]
is a transcendental meromorphic function.}

We claim that
\begin{enumerate}[(i)]
    \item the new map $g$ has a wandering domain $V$,
    \item $V$ is semi-contracting, and
    \item \textcolor{black}{$V$ is infinitely connected.}
\end{enumerate}

\textcolor{black}{The first claim will follow from the fact that $\psi$ conjugates $g_0$ to $g$, and that we left ``enough'' of $f$ intact in each $U_n$. More specifically, recall the round annulus $A$ from Subsection \ref{ssec:patient}; it satisfies $|B(z)| > c$ for $z\in A$. By Rouche's theorem, we find a topological annulus $A'\subset U_N$ (where, again, $N$ is large enough that the annuli $A_n'$ exist for $n\geq N$) such that, for $z\in A'$, $f^n(z)\in U_{n+N}$ does not intersect the discs $\{z\in \Cx : |z - 4(n + N)| < c'\}\subset U_{n+N}$, where $c'$ was also defined in Subsection \ref{ssec:patient}. In particular, for $z\in A'$, the $f$-orbit of $z$ is not affected by the surgery, and therefore is conjugated by $\psi$ to its $g$-orbit; it follows that $\psi(A')$ is contained in a Fatou component $V$ of $g$. Furthermore, since $g$ has a pole at $\psi(4n)$ for $n\geq N$, $V$ is at least doubly connected.}

Now, we show that $V$ is semi-contracting. We begin by showing that $V$ (and each $V_n$, the Fatou component of $g$ containing $g^n(V)$) is contained in $\psi(U_N)$ (respectively, $\psi(U_{n+N})$). This, in turn, will follow from the fact that we didn't modify $f$ outside of a small disc properly contained within each $U_n$. Indeed, if $w$ is a point on $\partial U_N$, it belongs to the Julia set of $f$, and is therefore approached by a sequence $(w_n)_{n\in\mathbb{N}}$ of repelling periodic points \textcolor{black}{of $f$ (see, for instance, \cite[Theorem 4]{Ber93})}. Since \textcolor{black}{the $w_n$ are periodic points in the Julia set, their orbits do not intersect the discs $\{z\in\Cx : |z - 4(n + N)| < c'\}\subset U_{n+N}$ for any $n\in\mathbb{N}$, and so} $g_0$ agrees with $f$ on \textcolor{black}{their $f$-orbits. It follows that} the conjugacy $\psi$ takes \textcolor{black}{these} $f$-orbits to corresponding $g$-periodic orbits $\psi(w_n)$, which \textcolor{black}{we can show to be repelling by the local topological dynamics as follows.}

\textcolor{black}{Take one of the $w_n$, of minimal period $k_n\geq 1$ (say), and apply Koenig's linearisation theorem \cite[Theorem 8.2]{Mil06} to find a neighbourhood $W$ of $w_n$ and a biholomorphic map $\phi:W\cup f^{k_n}(W)\to\phi(W\cup f^{k_n}(W))\subset\Cx$ such that $\phi$ conjugates $f^{k_n}|_W$ to multiplication by $\alpha_n := (f^{k_n})'(w_n)$, which satisfies $|\alpha_n| > 1$ since $w_n$ is repelling. This allows us to find (if necessary) a smaller neighbourhood $W'\subset W$ such that $f^m(W')\cap \{z\in\Cx : |z - 4(n + N)| < c'\} = \emptyset$ for all $n\in\mathbb{N}$ and $1\leq m \leq k_n$, so that $\psi\left(f^m(z)\right) = g^m\left(\psi(z)\right)$ for $z\in W'$ and $1\leq m\leq k_n$. From this, we may conclude that $\psi(w_n)$ is repelling for $g$: for all $z\in \psi(W')\setminus\{w_n\}$, there must exist $M\in\mathbb{N}$ such that $g^{Mk_n}(z)$ is not in $\psi(W')$ (see, for instance, \cite[p. 84]{Mil06}).} It follows that $\psi(w)$, being accumulated by repelling periodic points, belongs to $J(g)$\textcolor{black}{; a similar argument applies to each $V_n$, $n\in\mathbb{N}$}.

Thus, for any points $z, w\in V$, we have \textcolor{black}{by the Schwarz-Pick lemma and the fact that $V_n\subset\psi(U_{n+N})$ that}
\[ d_{V_n}\left(g^n(z), g^n(w)\right) \geq d_{\psi(U_{n+N})}\left(g^n(z), g^n(w)\right); \]
if $\psi$ were a conformal map, \textcolor{black}{our work would be done} -- but it is not, and it does not preserve the hyperbolic metric. Instead, for a domain $D\subset \Cx$ and points $z, w\in D$, let \textcolor{black}{us define}
\[ k_D(z, w) := \inf_\gamma \int_\gamma \frac{1}{d(s, \partial D)}\,|ds|, \]
where the infimum runs over every rectifiable arc $\gamma\subset D$ joining $z$ and $w$. This is the \textit{quasi-hyperbolic metric}; if $D$ is simply connected, standard estimates for the hyperbolic metric \cite[p. 13]{CG93} show that
\[ \frac{k_D(z, w)}{2} \leq d_D(z, w) \leq 2k_D(z, w). \]
Additionally, since $\psi$ is $(K_\infty)^{-1}$-Hölder continuous \textcolor{black}{\cite[p. 31]{BF14}}, we know (see \cite[Theorem 3]{GO79}) that there exists a constant $C > 0$ depending only on $K_\infty$ such that\textcolor{black}{
\[ k_{U_{n+N}}\left(z', w'\right) \leq C\cdot\max\left\{k_{\psi(U_{n+N})}\left(z, w\right), k_{\psi(U_{n+N})}\left(z, w\right)^{1/K_\infty}\right\}, \]
where $z' = \psi^{-1}(z)$ and $w' = \psi^{-1}(w)$ are points in $U_{n+N}$. For points $z'\in A'$, the construction of $g$ implies that
\[ \text{$\psi\left(f^n(z')\right) = g^n\left(\psi(z)\right)$ for all $n$,} \]
and so taking $z'$ and $w'$ in $A'$ yields
\[ k_{U_{n+N}}\left(f^n(z'), f^n(w')\right) \leq C\cdot\max\left\{k_{\psi(U_{n+N})}\left(g^n(z), g^n(w)\right), k_{\psi(U_{n+N})}\left(g^n(z), g^n(w)\right)^{1/K_\infty}\right\}. \]}
The left-hand side of this expression is bounded below by $d_{U_{n+N}}\left(f^n(z'), f^n(w')\right)/2$, which, since $U$ is semi-contracting, is in turn bounded below by $c(z', w')/2 > 0$. Also, since the exponent on the right-hand side is independent of $n$, we can get a uniform, positive lower bound $c'(z, w)$ on $k_{\psi(U_{n+N})}\left(g^n(z), g^n(w)\right)$ regardless of which term is the maximum. Combining everything, we see that
\[ \text{$\frac{c'(z, w)}{2} \leq d_{\psi(U_{n+N})}\left(g^n(z), g^n(w)\right) \leq d_{V_n}\left(g^n(z), g^n(w)\right)$ for $z$ and $w$ in $\psi(A')$,} \]
and thus $V$ contains a non-empty open subset that is semi-contracting relative to any $z_0$ in this subset -- we claim that $\left(d_{V_n}\left(g^n(z), g^n(w)\right)\right)_{n\in\mathbb{N}}$ is not a constant sequence, since each $g|_{V_n}$ ``inherits'' a critical point from $f$. Indeed, $f|_{U_n}$ has a single critical point $z_n^*$ for every $n$, and since $z_n^*$ approaches $\partial U_n$ as $n\to +\infty$ (in the sense outlined in property (D)), we have $z_n^*\in U_n\setminus\{z\in\Cx : |z - z_n| \leq r'\}$ for all sufficiently large $n$. Hence, the critical point is not affected by the surgery for sufficiently large $n$, and $\psi(z_n^*)\in V_n$ is a critical point of $g$ for all large $n$.

We have only shown that $V$ is semi-contracting on the topological annulus $\psi(A')$. For the remainder of $V$, we choose some $z_0$ in $\psi(A')$ and apply Theorem \ref{thm:mix}.

To prove that $V$ is infinitely connected, notice that (since $V_n\subset \psi(U_{n+N})$) all $V_n$ are bounded, and so $\deg g|_{V_n}$ is always finite. By applying \cite[Theorem 1.1]{Fer21}, we deduce that $V$ must be infinitely connected. This concludes the proof of Theorem~\ref{thm:ex}.

\end{document}